\documentclass[11pt]{article}
\usepackage{amsmath,amsthm}
\usepackage{amssymb,mathrsfs}
\usepackage{bm,mathtools}
\usepackage{tikz}
\usetikzlibrary{calc}

\usepackage{xcolor}

\usepackage{geometry}
\usepackage[colorlinks=true,
linkcolor=blue,citecolor=blue,
urlcolor=blue]{hyperref}

\makeatletter
\def\@seccntDot{.}
\def\@seccntformat#1{\csname the#1\endcsname\@seccntDot\hskip 0.5em}
\renewcommand\section{\@startsection{section}{1}{\z@}%
{18\p@ \@plus 6\p@ \@minus 3\p@}%
{9\p@ \@plus 6\p@ \@minus 3\p@}%
{\large\bfseries\boldmath}}
\renewcommand\subsection{\@startsection{subsection}{2}{\z@}%
{12\p@ \@plus 6\p@ \@minus 3\p@}%
{3\p@ \@plus 6\p@ \@minus 3\p@}%
{\bfseries\boldmath}}
\renewcommand\subsubsection{\@startsection{subsubsection}{3}{\z@}%
{12\p@ \@plus 6\p@ \@minus 3\p@}%
{\p@}%
{\bfseries\boldmath}}
\makeatother

\usepackage{microtype}

\theoremstyle{plain}
\newtheorem{theorem}{Theorem}[section]
\newtheorem{lemma}{Lemma}[section]

\newtheorem{proposition}{Proposition}[section]
\newtheorem{problem}{Problem}[section]
\newtheorem{conjecture}{Conjecture}[section]

\theoremstyle{definition}

\newtheorem{remark}{Remark}[section]
\newtheorem{example}{Example}[section]
\newtheorem{claim}{Claim}[section]

\numberwithin{equation}{section}
\allowdisplaybreaks
\DeclareMathOperator{\ex}{ex}
\DeclareMathOperator{\spex}{spex}

\DeclareMathOperator{\supp}{supp}

\newcommand{\imgi}{{\bf i}}
\newcommand{\e}{{\bf e}}

\title{Variants of spectral Tur\'an theorems and eigenvectors of graphs}

\author{Lele Liu\footnote{School of Mathematical Sciences, Anhui University, Hefei 230601, 
P.R. China. E-mail: \texttt{liu@ahu.edu.cn} (L. Liu). Supported by the National
Nature Science Foundation of China (No. 12001370)}
~~ and ~~
Bo Ning\footnote{Corresponding author. College of Computer Science, Nankai University, Tianjin 300350, P.R. China.
E-mail: \texttt{bo.ning@nankai.edu.cn} (B. Ning). Partially supported by the National Nature Science
Foundation of China (No. 11971346).}}

\date{}

\begin{document}
\maketitle

\begin{abstract}
In 2002, Nikiforov proved that for an $n$-vertex graph $G$ with clique number $\omega$ 
and edge number $m$, the spectral radius $\lambda(G)$ satisfies $\lambda (G) \leq \sqrt{2(1 - 1/\omega) m}$, 
which confirmed a conjecture implicitly suggested by Edwards and Elphick. In this paper, 
we prove a local version of spectral Tur\'an inequality, which states that 
$\lambda^2(G)\leq 2\sum_{e\in E(G)}\frac{c(e)-1}{c(e)}$, where $c(e)$ is the order of 
the largest clique containing the edge $e$ in $G$. We also characterize the extremal graphs. 
We prove that our theorem implies Nikiforov's theorem and give an example to show that 
the difference of Nikiforov's bound and ours is $\Omega (\sqrt{m})$ for some cases. 
Additionally, we establish a spectral counterpart to Ore's problem (1962) which asks for 
the maximum size of an $n$-vertex graph such that its complement is connected 
and does not contain $F$ as a subgraph. Our result leads to a new spectral Tur\'an 
inequality applicable to graphs with connected complements. Finally, we disprove a conjecture 
of Gregory, asserting that for a connected $n$-vertex graph $G$ with chromatic number 
$k\geq 2$ and an independent set $S$, we have 
\[ 
\sum_{v\in S} x_v^2 \leq \frac{1}{2} - \frac{k-2}{2\sqrt{(k-2)^2 + 4(k-1)(n-k+1)}},
\]
where $x_v$ is the component of the Perron vector of $G$ with respect to 
the vertex $v$. A modified version of Gregory's conjecture is proposed.
\par\vspace{2mm}

\par\vspace{2mm}

\end{abstract}

\section{Notation}

Given a graph $G$ of order $n$, the \emph{adjacency matrix} $A(G)$ of $G$ is an $n$-by-$n$ matrix 
whose $(i,j)$-entry is equal to $1$ if the vertices $i$ and $j$ are adjacent, and $0$ otherwise. 
Let $\lambda_1\geq \lambda_2\geq\cdots \geq \lambda_n$ be all eigenvalues of $A(G)$.
The \emph{spectral radius} of $G$ is the largest eigenvalue of $A(G)$, denoted by $\lambda (G)$.
Then $\lambda(G)=\lambda_1$. The Perron--Frobenius theorem for nonnegative matrices implies 
that $A(G)$ of a connected graph $G$ has a unique positive eigenvector of unit length corresponding 
to $\lambda(G)$, and this eigenvector is called the \emph{Perron vector} of $A(G)$. 
Throughout this paper, let $\chi(G)$ and $\omega(G)$ be the chromatic number and clique number 
of $G$, respectively, and $\chi$, $\omega$ in short. For a subset $X$ of the vertex set $V(G)$ 
of a graph $G$, we let $G[X]$ be the subgraph of $G$ induced by $X$, and denote by $e(X)$ 
the number of edges in $G[X]$. We also use $e(G)$ to denote the number of edges of $G$. 
As usual, for a vertex $v$ of $G$ we write $d_G(v)$ and $N_G(v)$ for the degree of $v$ 
and the set of neighbors of $v$ in $G$, respectively. If the underlying graph $G$ is 
clear from the context, simply $d(v)$ and $N(v)$. For graph notation 
and terminology undefined here, we refer the reader to \cite{Bondy-Murty2008}.

\section{A brief introduction to inequalities of spectral radius}

A complete characterization of bipartite graphs is that they contain no odd cycle. 
In spectral graph theory, one of the best-known theorems is that a graph containing 
at least one edge is bipartite if and only the spectrum is symmetric with respect 
to the zero point (see for example, Theorem 3.11 in \cite{Cvetkovi-Doob-Sachs1995}). 
This was largely extended by Hoffman \cite{Hoffman1970} to that $\chi\geq 1+\frac{\lambda_1}{-\lambda_n}$. 
Another spectral inequality but generally non-compared with Hoffman's bound is due 
to Cvetkovi\'c \cite{Cvetkovi1972}, which states $\chi\geq 1+\frac{\lambda_1}{n-\lambda_1}$. 
For certain graphs (e.g., $K_{10}$ minus an edge), Cvetkovi\'c's bound is better. 
By introducing the size $m$ of $G$, Edwards and Elphick \cite{Edwards-Elphick1983} 
proved that $\chi\geq 1+\frac{\lambda^2_1}{2m-\lambda^2_1}$ which is stronger than 
Cvetkovi\'c's bound.  

Another direction on the study of spectral inequality is to determine the spectral 
radius of a graph forbidden certain type of subgraphs. In 1970, Nosal \cite{Nosal1970} 
proved that any triangle-free graph $G$ satisfies that $\lambda(G)\leq \sqrt{m}$.
This theorem was extended to the general case by Nikiforov \cite{Nikiforov2002,Nikiforov2009}, which 
confirmed a conjecture originally due to Edwards and Elphick \cite{Edwards-Elphick1983}. Nikiforov's 
theorem is equivalent to that $\omega\geq 1+\frac{\lambda^2_1}{2m-\lambda^2_1}$, which is stronger 
than Edwards-Elphick bound, in views of the basic inequality of $\chi\geq \omega$. We write Nikiforov's 
bound in another form.

\begin{theorem}[Nikiforov \cite{Nikiforov2002,Nikiforov2009}]\label{thm:Spectral-radius-clique}
Let $G$ be a graph with $m$ edges. Then
\begin{equation}\label{eq:Spectral-radius-clique}
\lambda (G) \leq \sqrt{2\Big(1 - \frac{1}{\omega}\Big) m}.
\end{equation}
Moreover, equality holds if and only if $G$ is a complete bipartite graph for $\omega = 2$, 
or a complete regular $\omega$-partite graph for $\omega \geq 3$ with possibly some isolated vertices.
\end{theorem}

Nikiforov \cite{Nikiforov2020} pointed out that Theorem \ref{thm:Spectral-radius-clique} together with 
Tur\'an's theorem implies Stanley's inequality \cite{Stanley1987} that $\lambda(G)\leq -1/2+\sqrt{2m+1/4}$. 
It is noteworthy to highlight that Nikiforov's inequality implies the concise form of Tur\'an 
theorem \cite{Turan1961} and also Wilf's inequality \cite{Wilf1986}. Moreover, the first observation 
was pointed out by Aouchiche and Hansen earlier (see \cite[pp.~2297]{Aouchiche-Hansen2010}).

There are several conjectures which tried to extend Nikiforov's theorem. For details, see a 
conjecture of Bollob\'as-Nikiforov \cite[Conjecture~1]{Bollobas-Nikiforov2007} and a conjecture 
of Elphick-Linz-Wocjan \cite[Conjecture~2]{Elphick-Linze-Wocjan2023}. For advancement on the 
first conjecture, see \cite{Lin-Ning-Wu2021,Zhang2023}. For more related open problems and 
history, see Topics 1 and 2 in \cite{Liu-Ning2023}. However, Nikiforov's inequality 
(i.e., Theorem \ref{thm:Spectral-radius-clique}) is still the strongest one in some sense till now.

\section{A local version of spectral Tur\'an theorem}

The goal of this section is to prove a new spectral inequality stronger than Nikiforov's inequality. 
Before stating our result, we need some additional notation. Let $\triangle^{n-1}$ be the standard simplex:
\[
\triangle^{n-1}: = \bigg\{(x_1,x_2,\ldots,x_n) \in\mathbb{R}^n : x_i \geq 0, i\in [n] \ \text{and}\ \sum_{i=1}^n x_i = 1\bigg\}.
\]
Let $G$ be a graph on $n$ vertices with clique number $\omega (G)$. The well-known Motzkin--Straus 
Theorem \cite{Motzkin-Straus1965} states that
\begin{equation}\label{eq:Motzkin-Straus-inequality}
\max\{L_G(\bm{z}): \bm{z}\in\triangle^{n-1}\} = 1 - \frac{1}{\omega (G)},
\end{equation}
where $L_G(\bm{z}) = \bm{z}^{\mathrm{T}} A(G) \bm{z}$. Equation \eqref{eq:Motzkin-Straus-inequality} 
established a remarkable connection between the clique number and the Lagrangian of a graph. In fact, 
the connection can be employed to give a new proof of the fundamental theorem of Erd\H{o}s-Stone-Simonovits 
on Tur\'an densities of graphs \cite{Keevash2011}.

Additionally, Motzkin and Straus also gave a characterization on the equality 
$L_G(\bm{z}) = 1-1/\omega(G)$. Given a vector $\bm{x}$, the \emph{support} of 
$\bm{x}$, denoted by $\supp (\bm{x})$, is the set consisting of all indices 
corresponding to nonzero entries in $\bm{x}$.

\begin{proposition}[\cite{Motzkin-Straus1965}]\label{prop:MS-equality}
Let $G$ be an $n$-vertex graph with $\omega(G)=\omega$, and $\bm{z}\in\triangle^{n-1}$. 
Then $L_G(\bm{z})=1-1/\omega$ if and only if the induced subgraph of $G$ on 
$\supp (\bm{z})$ is a complete $\omega$-partite graph.
\end{proposition}

For a graph $G$ and an edge $e$ of $G$, let $c(e)$ denote the order of the largest 
clique containing $e$ in $G$. The main result of this section manifests in the following theorem.

\begin{theorem}\label{Thm:Main}
Let $G$ be a graph with spectral radius $\lambda(G)$ and $\omega(G)=\omega\geq 2$. Then 
\begin{equation}\label{eq:local-bound}
\lambda(G) \leq \sqrt{2 \sum_{e\in E(G)} \frac{c(e) - 1}{c(e)}}.
\end{equation}
Equality holds if and only if $G$ is a complete bipartite graph for $\omega = 2$, or a 
complete regular $\omega$-partite graph for $\omega \geq 3$ \textup{(}possibly with some isolated vertices\textup{)}.
\end{theorem}

\begin{proof}
We prove the assertion by a weighted Motzkin--Straus inequality. To this end, for any vector 
$\bm{x}\in\triangle^{n-1}$ we define 
\begin{equation}\label{eq:Fx}
F_G(\bm{x}):= \bm{x}^{\mathrm{T}} W\bm{x},
\end{equation}
where $W$ is an $n\times n$ matrix whose $(i,j)$-entry is $\frac{c(e)}{c(e) - 1}$ if 
$e = ij \in E(G)$, and $0$ otherwise. We first show that $F(\bm{x})$ is maximized at some 
point supported on a clique of $G$. This result has been proved in \cite{Bradac2022}, 
and \cite{Liu-Reiher2021} in the context of a related setting. For the sake of 
completeness, we provide a proof here.

\begin{claim}\label{claim:support-clique}
Let $\bm{x}\in\triangle^{n-1}$ be a vector such that 
$F_G(\bm{x}) = \max\{F_G(\bm{z}): \bm{z}\in\triangle^{n-1}\}$ with minimal supporting set.
Then $\supp (\bm{x})$ induces a clique in $G$.
\end{claim}

\noindent\textit{Proof of Claim \ref{claim:support-clique}}.
Suppose by contradiction that there exists $\bm{x}\in\triangle^{n-1}$ such that 
$F_G(\bm{x}) = \max\{F_G(\bm{z}): \bm{z}\in\triangle^{n-1}\}$ and $\bm{x}$ has a 
minimal supporting set, but there exist $i\neq j$ such that $ij\notin E(G)$ and 
$x_i, x_j > 0$. 
Without loss of generality, assume that $\bm{x}^{\mathrm{T}} W \bm{e}_i\geq \bm{x}^{\mathrm{T}} W \bm{e}_j$,
where $\bm{e}_i$ and $\bm{e}_j$ are the $i$-th column and $j$-th column of the 
identity matrix of order $n$, respectively. 

Consider the vector $\bm{y} = \bm{x} + x_j(\bm{e}_i - \bm{e}_j)$, and set 
$\bm{e}:= \bm{e}_i - \bm{e}_j$, it is clear that $\bm{y}\in\triangle^{n-1}$. By simple algebra, 
we have 
\[
F_G(\bm{y}) - F_G(\bm{x}) = \bm{y}^{\mathrm{T}} W \bm{y} - \bm{x}^{\mathrm{T}} W \bm{x} 
= 2x_j \bm{x}^{\mathrm{T}} W \bm{e} + x_j^2 \bm{e}^{\mathrm{T}} W \bm{e}. 
\]
Since $ij\notin E(G)$, we see $\bm{e}^{\mathrm{T}} W \bm{e} = 0$. So,
\[
F_G(\bm{y}) - F_G(\bm{x}) = 2x_j\bm{x}^{\mathrm{T}} W \bm{e} = 2x_j \bm{x}^{\mathrm{T}} W (\bm{e}_i - \bm{e}_j) \geq 0.
\]
Hence, $\bm{y}$ is also a vector attaining $\max\{F_G(\bm{z}): \bm{z}\in\triangle^{n-1}\}$.
However, $\supp (\bm{y}) < \supp (\bm{x})$, contradicting our assumption on $\bm{x}$.
\hfill $\diamondsuit$ \par\vspace{2.5mm}

\begin{claim}\label{claim:eigenequation}
Let $\bm{x}\in\triangle^{n-1}$ such that $F_G(\bm{x})=\max\{F_G(\bm{z}): \bm{z}\in\triangle^{n-1}\}$. 
Then $\bm{x}^{\mathrm{T}} W\bm{e}_i = \bm{x}^{\mathrm{T}} W\bm{e}_j$ for any $i,j\in\supp (\bm{x})$.
\end{claim}

\noindent\textit{Proof of Claim \ref{claim:eigenequation}}.
Let $i,j\in\supp (\bm{x})$. We assume, without loss of generality, towards contradiction that 
$\bm{x}^{\mathrm{T}} W\bm{e}_i > \bm{x}^{\mathrm{T}} W\bm{e}_j$. 
Let $0<\varepsilon<\min\{x_i, x_j\}$. Denote $\bm{y} := \bm{x} + \varepsilon (\bm{e}_i - \bm{e}_j)$. Then
\begin{align*}
F_G(\bm{y}) - F_G(\bm{x}) 
& = 2\varepsilon \bm{x}^{\mathrm{T}} W(\bm{e}_i - \bm{e}_j) + \varepsilon^2 (\bm{e}_i-\bm{e}_j)^{\mathrm{T}} W (\bm{e}_i - \bm{e}_j) \\
& = 2\varepsilon \bm{x}^{\mathrm{T}} W(\bm{e}_i - \bm{e}_j) - 2\varepsilon^2 \bm{e}_i^{\mathrm{T}} W \bm{e}_j.
\end{align*}
Choosing $\varepsilon$ to be sufficiently small gives $F_G(\bm{y}) > F_G(\bm{x})$, 
a contradiction.\hfill $\diamondsuit$ \par\vspace{2.5mm}

\begin{claim}\label{claim:max-Fz}
For each $\bm{z}\in\triangle^{n-1}$, $F_G(\bm{z}) \leq 1$. Equality holds if and only if 
the induced subgraph of $G$ on $\supp (\bm{z})$ is a complete $\omega$-partite graph.
\end{claim}

\noindent\textit{Proof of Claim \ref{claim:max-Fz}}.
Let $\bm{z}$ be a vector satisfying $F_G(\bm{z}) = \max\{F(\bm{y}): \bm{y}\in\triangle^{n-1}\}$ 
with minimal supporting set. By Claim \ref{claim:support-clique}, $\supp (\bm{z})$ forms a 
clique $K$ in $G$. Hence,
\[
F_G(\bm{z}) = 2 \sum_{e\, =\, ij\in E(K)} \frac{c(e)}{c(e) - 1} z_iz_j.
\]
Let $r$ be the order of $K$. Observe that $x/(x-1)$ is a decreasing function in $x$, 
we immediately obtain 
\[
F_G(\bm{z}) \leq \frac{2r}{r - 1} \sum_{ij\in E(K)} z_iz_j
\leq 1,
\]
where the last inequality follows by applying \eqref{eq:Motzkin-Straus-inequality} to the clique $K$.

In the following we characterize the equality of $F_G(\bm{z}) = 1$. If the induced subgraph of 
$G$ on $\supp (\bm{z})$ is a complete $\omega$-partite graph, then $L_G(\bm{z}) = 1-1/\omega$ 
by Proposition \ref{prop:MS-equality}. Since $c(e) \leq\omega$ for each $e\in E(G)$, we have   
\begin{align*}
F_G(\bm{z}) 
& = 2\sum_{ij\in E(G)} \frac{c(e)}{c(e)-1} z_iz_j \\
& \geq \frac{2\omega}{\omega-1} \sum_{ij\in E(G)} z_iz_j \\
& = \frac{\omega}{\omega-1} \cdot L_G(\bm{z}) \\
& = 1,
\end{align*}
which yields that $F_G(\bm{z})=1$.

For the other direction, let $F_G(\bm{z})=1$, we will show that the induced subgraph of $G$ 
on $\supp (\bm{z})$ is a complete $\omega$-partite graph by induction on the order $n$ of $G$. 
The case of $n=2$ is clear. Now let $n\geq 3$. If $G[\supp(\bm{z})]$ is a complete graph, then 
the assertion is true. So we assume $v_1,v_2\in\supp(\bm{z})$ and $v_1v_2\notin E(G)$. 
Let $G':=G-v_1$, and define a vector $\bm{z}'\in\mathbb{R}^{n-1}$ for $G'$ by 
\[
z'_v = 
\begin{cases}
z_v, & v\neq v_2, \\
z_{v_1} + z_{v_2}, & v = v_2.
\end{cases}
\]
Then $F_{G'}(\bm{z}') = F_G(\bm{z}) = 1$ by Claim \ref{claim:eigenequation}. By induction, 
the induced subgraph $H$ of $G'$ on $\supp (\bm{z}')$ is a complete $k$-partite graph, where 
$k=\omega(G')$. Let $V_1,\ldots,V_k$ be the vertex classes of $H$. If $N_G(v_1)\cap V_i \neq\emptyset$ 
for each $i=1,2,\ldots,k$, then $\omega(G) = k+1$. Hence, $F_G(\bm{z}) > 1$, a contradiction.
So, there is some $V_i$, say $V_1$, such that $v_1$ is nonadjacent to any vertices in $V_1$. 
However, since $F_G(\bm{z})=\max\{F_G(\bm{x}): \bm{x}\in\triangle^{n-1}\}$, for each $v\in V_1$ 
we have $\bm{z}^{\mathrm{T}} W\bm{e}_{v_1} = \bm{z}^{\mathrm{T}} W\bm{e}_v$ by Claim \ref{claim:eigenequation}. 
It follows that $v_1$ is adjacent to all vertices in $V_2\cup V_3\cup\cdots\cup V_k$, completing the proof.
\hfill $\diamondsuit$ \par\vspace{2.5mm}

Now, let $\bm{z}$ be the Perron vector of $A(G)$, and $\bm{y}$ be the vector such that 
$y_i = z_i^2$ for $i\in [n]$. Then
\begin{align*}
\lambda (G) 
& = 2 \sum_{ij\in E(G)} z_iz_j = 2 \sum_{ij\in E(G)} \sqrt{y_iy_j} \\
& = 2 \sum_{e\,=\, ij\in E(G)} \sqrt{\frac{c(e) - 1}{c(e)}} \cdot \sqrt{\frac{c(e)}{c(e) - 1} y_iy_j}.
\end{align*}
Using Cauchy--Schwarz inequality and Claim \ref{claim:max-Fz}, we find
\begin{align*}
\lambda^2(G)
& \leq 4 \sum_{e\in E(G)} \frac{c(e) - 1}{c(e)} \cdot \sum_{e\,=\, ij\in E(G)} \frac{c(e)}{c(e) - 1} y_iy_j \\
& = 2 \bigg( \sum_{e\in E(G)} \frac{c(e) - 1}{c(e)} \bigg) \cdot F_G(\bm{y}) \\
& \leq 2 \sum_{e\in E(G)} \frac{c(e) - 1}{c(e)},
\end{align*}
completing the proof of \eqref{eq:local-bound}.

Finally, we characterize all the graphs attaining equality in \eqref{eq:local-bound}. According to 
the proof above, equality in \eqref{eq:local-bound} holds if and only if 
$F_G(\bm{y})=1$ and meantime, for each $e=ij\in E(G)$,
\[ 
y_iy_j = c\left(\frac{c(e)-1}{c(e)}\right)^2 
\]
for some constant $c>0$. By Claim \ref{claim:max-Fz} and $y_i>0$ for any $i\in V(G)$, this equivalent to 

(1). $G$ is a compete $\omega$-partite graph;

(2). $y_iy_j = c'$ for some constant $c'>0$.

\noindent This completes the proof for $\omega=2$.
Let now $\omega\geq 3$, and $G$ be a compete $\omega$-partite graph.
For any $i\neq j$ and $u\in V_i, v\in V_j$, by item (2), we have $y_uy_w=c'=y_vy_w$ for 
any $w\notin (V_i\cup V_j)$. Then $z_u=z_v$, and therefore $|V_1|=|V_2|=\cdots=|V_{\omega}|$, 
completing the proof.
\end{proof}

Our theorem implies Nikiforov's inequality.

\begin{proof}[Proof of Theorem \ref{thm:Spectral-radius-clique}] 
Let $G$ be a graph with $\omega(G)=\omega$. Then $G$ is $K_{\omega+1}$-free, 
and so for any edge $e\in E(G)$, $c(e)\leq \omega$ and
$\frac{c(e)-1}{c(e)}\leq 1-\frac{1}{\omega}$. By Theorem \ref{Thm:Main}, we 
have $\lambda(G)\leq \sqrt{2m\cdot(1-1/\omega})$. The proof is complete.
\end{proof}

Recall that a \emph{kite graph} $Ki_{m,\omega}$ is a graph with $m$ edges obtained 
from a clique $K_{\omega}$ and a path by adding an edge between a vertex from the 
clique and an endpoint from the path.

\begin{example}
Let $\omega\geq 3$ be a fixed integer, and $Ki_{m,\omega}$ be the kite graph with 
clique number $\omega$. By simple algebra,
\[
2\sum_{e\in E(Ki_{m,\omega})} \frac{c(e)-1}{c(e)}
= m + \binom{\omega-1}{2}.
\]
Hence, for $m>\omega\binom{\omega-1}{2}$, the difference between Nikiforov's 
bound \eqref{eq:Spectral-radius-clique} and our bound \eqref{eq:local-bound} is
\begin{align*}
\sqrt{2\Big(1-\frac{1}{\omega}\Big) m} - \sqrt{2\sum_{e\in E(Ki_{m,\omega})} \frac{c(e)-1}{c(e)}} 
& = \sqrt{2\Big(1-\frac{1}{\omega}\Big) m}  - \sqrt{m + \binom{\omega-1}{2}} \\
& > \sqrt{2\Big(1-\frac{1}{\omega}\Big) m}  - \sqrt{\Big(1+\frac{1}{\omega}\Big) m} \\
& = \Omega (\sqrt{m}).
\end{align*}
\end{example}

Concluding this section, we explore an analog of \eqref{eq:Spectral-radius-clique} concerning 
the signless Laplacian spectral radius of graphs. To recap, the signless Laplacian matrix of $G$ 
is defined as $Q(G) := D(G) + A(G)$, where $D(G)$ is the diagonal matrix whose entries are 
the degrees of the vertices of $G$. The largest eigenvalue of $Q(G)$, denoted by $q(G)$, is 
called the \emph{signless Laplacian spectral radius} of $G$. Li, Feng, and Liu \cite{Li-Feng-Liu2022} 
presented the following conjecture.

\begin{conjecture}[\cite{Li-Feng-Liu2022}]\label{conj:q-clique}
Let $G$ be a $K_{r+1}$-free graph with $m$ edges. Then 
\[
q(G) \leq \sqrt{8\Big( 1-\frac{1}{r} \Big) m}.
\]
\end{conjecture}

Conjecture \ref{conj:q-clique} is not true in general. Consider the star $S_n$
on $n$ vertices where $n > 8(1-1/r)$. One can check that 
$q(S_n) = n > \sqrt{8( 1-1/r) n} > \sqrt{8( 1-1/r)e(S_n)}$.
Generally speaking, if $G$ is a $K_{r+1}$-free graph on $n$ vertices 
with $\Delta (G) = \Omega (n)$ and $e(G) = o(n^2)$, then $G$ is always 
a counterexample to Conjecture \ref{conj:q-clique} for sufficiently large $n$.

\section{A variation of spectral Tur\'an problem}

In this section we explore a variation of spectral Tur\'an problem, serving as 
a spectral analogue to Ore's problem.

A graph $G$ satisfying that the complement $\overline{G}$ of $G$ is connected is 
called \emph{co-connected}. Let $\ex_{cc}(n,F)$ denote the maximum size of a 
co-connected graph $G$ on $n$ vertices which does not contain $F$ as a subgraph. 
If the connectedness on $\overline{G}$ is removed, then $\ex_{cc}(n,F)$ becomes 
the Tur\'an number $\ex (n, F)$ of $G$. An old problem of Ore \cite{Ore1962},
proposed in 1962, sought to determine $\ex_{cc}(n, K_{r+1})$. This problem has been settled 
independently by Bougard-Joret \cite{BJ08}, Christophe et al. \cite{Christophe2008}, 
Gitler-Valencia \cite{Gitler-Valencia2013} and Yuan \cite{Y19}. It was proved 
in \cite{BJ08,Christophe2008,Gitler-Valencia2013,Y19} that $\ex_{cc}(n,K_{r+1}) = e(T_r(n)) - r + 1$ 
for $n\geq r+1$. For some advancement on a similar problem on matchings, see \cite{Shi-Ma2024}.

As a spectral counterpart, it is natural to consider the following problem.

\begin{problem}
What is the maximum spectral radius $\spex_{cc}(n, F)$ among all $F$-free co-connected graphs of order $n$?
\end{problem}

In this section, we will determine the unique graph attaining the maximum spectral radius among all 
$K_{r+1}$-free co-connected graphs of order $n$. 

Let us begin with some definitions. Consider $r$ positive integers $n_1,n_2,\ldots,n_r$ such that 
$n_1+n_2+\cdots+n_r=n$. The complete $r$-partite graph of order $n$ with vertex classes $V_1,V_2,\ldots,V_r$
is denoted by $K_{n_1,n_2,\ldots,n_r}$, where $|V_i|=n_i$, $i=1,2,\ldots,r$. In the following, 
we assume $n_1\geq n_2\geq\cdots\geq n_r$. The Tur\'an graph $T_r(n)$ is the complete 
$r$-partite graph $K_{n_1,n_2,\ldots,n_r}$ satisfying $n_1-n_r\leq 1$. Let $T_r^-(n)$ denote 
the graph obtained from $T_r(n)$ by deleting a star $S_r$ on $r$ vertices, where the center of 
the star is located within $V_1$.

The main result of this section is as follows.

\begin{theorem}\label{thm:co-connected-spectral}
Let $r\geq 2$ and $n\geq 65r^2$. Let $G$ be a co-connected $K_{r+1}$-free graph of order $n$. 
Then $\lambda(G)\leq\lambda(T_r^{-}(n))$, with equality if and only if $G=T_r^{-}(n)$.   
\end{theorem}

\begin{remark}
In our proof, we make no attempt to determine the smallest possible value of $n$ for which 
the theorem is valid, although it can be worked out from the proof to be less than $65r^2$.
\end{remark}

Before giving the proof of Theorem \ref{thm:co-connected-spectral}, we need the following two 
results.

\begin{proposition}[\cite{Mubayi2010}]\label{prop:almost-equal-size}
Suppose that $r\geq 2$ is fixed, $s<n$ and $n_1 +\cdots + n_r = n$. If 
\[
\sum_{1\leq i<j\leq r} n_in_j \geq e(T_r(n)) - s,
\]
then $\lfloor n/r\rfloor - s\leq n_i\leq\lceil n/r\rceil + s$ for all $i$.
\end{proposition}

\begin{lemma}[\cite{Brouwer1981}]\label{lem:r-partite}
Let $r\geq 2$ and $n\geq 2r+1$. Every $K_{r+1}$-free graph with at least 
$e(T_r(n)) - \lfloor n/r\rfloor + 2$ edges is $r$-partite.
\end{lemma}

We are now ready to complete the proof of Theorem \ref{thm:co-connected-spectral}.
\par\vspace{2mm}

\noindent {\it Proof of Theorem \ref{thm:co-connected-spectral}}.
Let $G$ be a co-connected $K_{r+1}$-free graph of order $n$ that attains the maximum 
spectral radius. It suffices to show that $G=T_r^{-}(n)$. To finish the proof we 
first prove a few facts about $G$ and the Perron vector $\bm{x}$ of $G$.

\begin{claim}\label{claim:lower-bound-lambda}
$\lambda(G) > \lambda(T_r(n)) - 2r/n$.
\end{claim}

\noindent {\it Proof of Claim \ref{claim:lower-bound-lambda}}.
Let $\bm{y}$ be the Perron vector of $T_r(n)$. Then for each $v\in V(T_r(n))$,
\[
y_v < \frac{1}{\sqrt{n-r}}.
\]
By the maximality of $\lambda(G)$ and Rayleigh's principle,
\begin{align*}
\lambda(G) \geq \lambda(T_r^-(n)) 
& \geq 2\sum_{uv\in E(T_r^-(n))} y_uy_v \\
& = 2\sum_{uv\in E(T_r(n))} y_uy_v - 2\sum_{uv\in E(T_r(n))\setminus E(T_r^-(n))} y_uy_v \\
& > \lambda(T_r(n)) - \frac{2(r-1)}{n-r} \\
& > \lambda(T_r(n)) - \frac{2r}{n},
\end{align*}
completing the proof of the claim.
\hfill $\diamondsuit$ \par\vspace{2.5mm}

\begin{claim}\label{claim:large-size}
$e(G)\geq e(T_r(n)) - 17r/8$.
\end{claim}

\noindent {\it Proof of Claim \ref{claim:large-size}}.
By Claim \ref{claim:lower-bound-lambda}, $\lambda(G) > \lambda(T_r(n)) - 2r/n$.
On the other hand, by Theorem \ref{thm:Spectral-radius-clique},
\[
\lambda(G) \leq\sqrt{2\Big(1-\frac{1}{r}\Big)e(G)}.
\]
Combining the two inequalities, gives 
\[
2\Big(1-\frac{1}{r}\Big)e(G) > \Big(\lambda (T_r(n)) - \frac{2r}{n}\Big)^2
\geq \frac{(2e(T_r(n)) - 2r)^2}{n^2}.
\]
Solving this inequality, we find that
\begin{equation}\label{eq:e(G)}
e(G) > \frac{2r}{(r-1)n^2} \cdot (e(T_r(n)) - r)^2.
\end{equation}

Let $n = kr + t$, $0\leq t < r$. A simple calculation shows that
\begin{equation}\label{eq:size-Turan-graph}
e(T_r(n)) = \frac{1}{2} \Big( 1 - \frac{1}{r} \Big) n^2 - \frac{1}{2} t \Big( 1 - \frac{t}{r} \Big).
\end{equation} 
In view of \eqref{eq:e(G)} and \eqref{eq:size-Turan-graph},
\begin{align*}
e(G) 
& \geq e(T_r(n)) - 2r - \frac{t}{2} \Big(1-\frac{t}{r}\Big) \\
& \geq e(T_r(n)) - \frac{17r}{8},
\end{align*}
completing the proof of the claim.
\hfill $\diamondsuit$ \par\vspace{2.5mm}

By Lemma \ref{lem:r-partite} and Claim \ref{claim:large-size}, $G$ must be an $r$-partite graph.
Hereafter, assume that $G$ is an $r$-partite graph with vertex classes $V_i$, where $|V_i|=n_i$ 
for $i=1,2,\ldots,r$ and $n_1\geq n_2\geq\cdots\geq n_r$.

Prior to proceeding further, let us address the case when $r = 2$. If $n$ is even, similar to the reasoning 
in Claim \ref{claim:lower-bound-lambda} and Claim \ref{claim:large-size}, we find that $e(G) = e(T_2(n)) - 1$, 
and therefore $G=T_2^-(n)$. If $n$ is odd, it follows that $\lambda (G)\geq\sqrt{e(T_2(n))} - 2/(n-2)$, leading to 
$e(G)\geq e(T_2(n))-2$, in alignment with the arguments presented in Claim \ref{claim:lower-bound-lambda} and 
Claim \ref{claim:large-size}. On the other hand, $e(G)\leq n_1(n-n_1) - 1$. Consequently, $n_1(n-n_1) \geq e(T_2(n)) - 1$, 
which yields $n_1=\lceil n/2\rceil$. Hence, $G = T_2^-(n)$.

In what follows, we assume $r\geq 3$. We define a graph $T_G$ associated with $G$ and $\{V_i\}_{i=1}^r$ as 
follows: the vertex set of $T_G$ is $\{V_1,V_2,\ldots,V_r\}$, and $V_i$ and $V_j$ are adjacent in $T_G$ 
if and only if there is an edge between $V_i$ and $V_j$ in $\overline{G}$. The next claim is clear, 
and is therefore presented here without a proof.

\begin{claim}\label{claim:iff-co-connected}
$G$ is co-connected if and only if $T_G$ is connected.
\end{claim}

By Claim \ref{claim:iff-co-connected} and the maximality of $\lambda(G)$, $T_G$ is a tree, and for 
any $i\neq j$ there is at most one non-edge between $V_i$ and $V_j$. It follows that $e(G) = e(K_{n_1,n_2,\ldots,n_r}) - (r-1)$.

\begin{claim}\label{claim:T_G-star}
$T_G$ is a star.
\end{claim}

\noindent {\it Proof of Claim \ref{claim:T_G-star}}.
The assertion is clear for $r=3$. So, in the following we assume $r\geq 4$ and prove this claim 
by contradiction. To do this, we assume, contrary to our claim, that $T_G$ is not a star. 
Then the diameter of $G$ is at least $3$ as $G$ being a tree. Without loss of generality, 
let $V_1V_2V_3V_4$ be a path of length $3$ in $T_G$ (recall that $V_1$, $V_2$, $V_3$, $V_4$ 
are the vertices of $T_G$). Hence, there exist vertices $v_1\in V_1$, $v_4\in V_4$ and 
$v_j^{(1)}, v_j^{(2)}\in V_j$, $j=2,3$ such that $v_1v_2^{(1)}, v_2^{(2)} v_3^{(1)}, v_3^{(2)} v_4\notin E(G)$.

We first show that $v_2^{(1)}$ and $v_2^{(2)}$ are the same vertex. Indeed, if not, 
we may assume $x_{v_2^{(1)}}\geq x_{v_2^{(2)}}$. Let
$G_1:= G + v_1v_2^{(1)} - v_1v_2^{(2)}$. Clearly, $G_1$ is co-connected and 
$\lambda (G_1) > \lambda (G)$, a contradiction. Likewise, $v_3^{(1)} = v_3^{(2)}$. 

Hereafter, we denote $v_j:= v_j^{(i)}$ for $j=2,3$. Since $T_G$ is a tree, $v_2v_4\in E(G)$. 
We can assume, without loss of generality, that $x_{v_2}\leq x_{v_3}$. Let $G_2:=G+v_3v_4 - v_2v_4$. 
Clearly, $G_2$ is co-connected by Claim \ref{claim:iff-co-connected}. However, 
$\lambda(G_2)>\lambda(G)$, a contradiction finishing the proof of the claim.
\hfill $\diamondsuit$ \par\vspace{2.5mm}

By Claim \ref{claim:T_G-star} and the maximality of $\lambda(G)$, $G$ is an $r$-partite 
graph obtained from $K_{n_1,n_2,\ldots,n_r}$ by removing a star $S_r$. For each 
$i=1,2,\ldots,r$, let $w_i\in V_i$ denote the ends of the star. For convenience, 
we also use $w$ to denote the center of the star and $K_{n_1,n_2,\ldots,n_r}^w$ to 
denote $G$ henceforth. Additionally, set $\lambda:=\lambda (G)$ for short.

\begin{claim}\label{claim:x-comp}
For each $v\in V(G)$,
\[ 
\frac{1}{\sqrt{n}} - \frac{10r}{\lambda \sqrt{n}} < x_v < \frac{1}{\sqrt{n}} + \frac{10r}{\lambda \sqrt{n}}.
\]
\end{claim}

\noindent {\it Proof of Claim \ref{claim:x-comp}}.
Let $v\in V(G)$. We first give a rough upper bound on $x_v$. By eigenvalue\,--\,eigenvector 
equation and Cauchy\,--\,Schwarz inequality, 
\[
(1 + \lambda) x_v = x_v + \sum_{uv\in E(G)} x_u \leq \sqrt{n}.
\]
On the other hand, in view of Claim \ref{claim:lower-bound-lambda}, 
\begin{align*}
\lambda  
& > \lambda (T_r(n)) - \frac{2r}{n} \geq \Big(n -\Big\lceil\frac{n}{r}\Big\rceil\Big) - \frac{2r}{n} \\
& \geq n - \Big(\frac{n}{r} + 1 - \frac{1}{r}\Big) - \frac{2r}{n} \\
& > \Big(1 - \frac{1}{r}\Big) n - 1.
\end{align*}
As a consequence,
\begin{equation}\label{eq:x-comp-upper-bound}
x_v \leq \frac{\sqrt{n}}{1 + \lambda} < \frac{r}{r-1} \frac{1}{\sqrt{n}}.
\end{equation}

Applying \eqref{eq:x-comp-upper-bound}, further estimations on $x_v$ can be derived.
To this end, note that $\lfloor n/r\rfloor - 17r/8 \leq n_i\leq\lceil n/r\rceil + 17r/8$ 
by Proposition \ref{prop:almost-equal-size} and Claim \ref{claim:large-size}.
Taking into account \eqref{eq:x-comp-upper-bound} along with the eigenvalue\,--\,eigenvector 
equation, we arrive at the following inequality:
\begin{equation}\label{eq:xmax-xmin}
x_{\max} - x_{\min} < \frac{10r}{\lambda \sqrt{n}},
\end{equation}
where $x_{\max}$ and $x_{\min}$ are the maximum and the minimum of the components of $\bm{x}$, respectively.
Recall that $\bm{x}$ has length $1$. Hence, 
\[
1 = \|\bm{x}\|_2^2 \geq nx_{\min}^2 > n\Big(x_{\max} - \frac{10r}{\lambda \sqrt{n}}\Big)^2,
\]
which yields that
\[
x_{\max} < \frac{1}{\sqrt{n}} + \frac{10r}{\lambda \sqrt{n}}.
\]
On the other hand, by \eqref{eq:xmax-xmin},
\[
x_{\min} > x_{\max} - \frac{10r}{\lambda \sqrt{n}} \geq \frac{1}{\sqrt{n}} - \frac{10r}{\lambda \sqrt{n}},
\]
completing the proof of Claim \ref{claim:x-comp}.
\hfill $\diamondsuit$ \par\vspace{2.5mm}

Next, we will demonstrate through contradiction that $n_1 - n_r \leq 1$. We will argue that if 
$n_1-n_r\geq 2$ then $\lambda (K_{n_1-1,n_2,\ldots,n_{r-1},n_r+1}^w) > \lambda (G)$, leading to a contradiction.

Let $a_i$ denote the $\bm{x}$-component of the vertex $w_i$, and $b_i$ denote the 
$\bm{x}$-component of the vertices in $V_i$ rather than $w_i$, $i=1,2,\ldots,r$.
By eigenvalue\,--\,eigenvector equation, 
\[
\lambda b_i = \sum_{j\neq i} (n_j-1)b_j + \sum_{j=1}^r a_j - a_i
=: A + B - a_i - (n_i-1) b_i,
\]
where $A:=\sum_{j=1}^r a_j$, $B:=\sum_{j=1}^r (n_j-1) b_j$. This is equivalent to
\begin{equation}\label{eq:eq1}
(\lambda + (n_i-1)) b_i = A + B - a_i, \quad i=1,2,\ldots,r.
\end{equation}
Setting $K:=K_{n_1-1,n_2,\ldots,n_{r-1},n_r+1}^w$ for short and using Rayleigh's principle, we deduce that
\begin{align*}
\lambda (K) & \geq \bm{x}^{\mathrm{T}} A(K) \bm{x} \\
& = 2\sum_{uv\in E(G)} x_ux_v + 2\sum_{uv\in E(K)\setminus E(G)} x_ux_v - 2\sum_{uv\in E(G)\setminus E(K)} x_ux_v \\
& = \lambda (G) + 2b_1 (a_1 + (n_1-2)b_1) - 2b_1(a_r + (n_r-1)b_r).
\end{align*}

In the following we will show $a_1 + (n_1-2)b_1 > a_r + (n_r-1)b_r$, and therefore $\lambda (K) > \lambda (G)$.
Towards this objective, note that 
\begin{align*}
a_1 + (n_1-2)b_1 - (a_r + (n_r-1)b_r)
& = \lambda b_r - (\lambda + 1) b_1
\end{align*}
by \eqref{eq:eq1}. Thus, in order to derive a contradiction, it is sufficient to show that $\lambda b_r - (\lambda + 1) b_1 > 0$.
According to \eqref{eq:eq1} and Claim \ref{claim:x-comp}, we see
\begin{equation}\label{eq:auxiliary}
(\lambda + (n_1-1)) b_1 - (\lambda + (n_r-1)) b_r = a_r-a_1 < \frac{20r}{\lambda \sqrt{n}}.
\end{equation}
Since $n_1-n_r\geq 2$, by Claim \ref{claim:x-comp} and \eqref{eq:auxiliary} we have
\begin{align*}
(\lambda + 1) b_1 
& < \frac{(\lambda+1)(\lambda+n_r-1)}{\lambda+n_1-1} b_r 
+ \frac{20r (\lambda + 1)}{\lambda (\lambda + n_1 - 1) \sqrt{n}} \\
& \leq \frac{(\lambda+1)(\lambda+n_1-3)}{\lambda+n_1-1} b_r 
+ \frac{20r (\lambda + 1)}{\lambda (\lambda + n_1 - 1) \sqrt{n}} \\
& = \lambda b_r - \frac{1}{\lambda+n_1-1} \bigg((\lambda-n_1+3) b_r - \frac{20r (\lambda + 1)}{\lambda \sqrt{n}}\bigg) \\
& < \lambda b_r.
\end{align*}
Hence, $\lambda (K) > \lambda (G)$. This contradiction implies that $n_1-n_r\leq 1$.

Finally, it remains to show that $w$ (the center of the star that is removed from $T_r(n)$) 
belongs to the vertex class of size $\lceil n/r\rceil$, which implies $G = T_r^{-}(n)$. 
If $n$ is divisible by $r$, there is no need to prove. If $n$ is not divisible by $r$, 
it suffices to show that $\lambda (T_r^{-}(n)) > \lambda (H)$, where $H$ is obtained 
from $T_r(n)$ by removing a star $S_r$ and the center of the star is in $V_r$. 
According to the proof above, $n_1=\lceil n/r\rceil$, $n_r=\lfloor n/r\rfloor$. 

To continue, we first introduce some notation. Let $\bm{y}$ and $\bm{z}$ denote the Perron vectors of
$G':=T_r^{-}(n)$ and $H$, respectively. For each $i=1,2,\ldots,r$, let $y_i$ and $z_i$ be the 
$\bm{y}$- and $\bm{z}$-components of the vertex $w_i$, respectively, and let $h_i$ and $k_i$ 
be the $\bm{y}$- and $\bm{z}$-components of the other vertices in $V_i$ rather than $w_i$.  

Using eigenvalue\,--\,eigenvector equation, $\lambda(G') (h_i - y_i) = y_1$ for $i=2,\ldots,r$.
It follows that 
\begin{equation}\label{eq:intermediate-1}
h_{i+1} - h_i = y_{i+1} - y_i, \quad i =2,\ldots,r-1.
\end{equation}
By \eqref{eq:intermediate-1} and eigenvalue\,--\,eigenvector equation,
\begin{align*}
(\lambda (G') + 1) (y_{i+1} - y_i) 
& = (n_i - 1) h_i - (n_{i+1} - 1) h_{i+1} \\
& \geq (n_{i+1} - 1) (h_i - h_{i+1}) \\
& = (n_{i+1} - 1) (y_i - y_{i+1}),
\end{align*}
which yields that $y_r\geq y_{r-1}\geq\cdots\geq y_2$. Likewise, we have
$z_1\leq z_2\leq\cdots\leq z_{r-1}$. It follows that 
\begin{align}\label{eq:lambda-G-H}
\begin{split}
(\lambda (G')-\lambda (H)) \bm{y}^{\mathrm{T}}\bm{z}
& = \bm{y}^{\mathrm{T}} (A(T_r^{-}(n)) - A(H)) \bm{z} \\
& = (y_r-y_1) (z_2+\cdots + z_{r-1}) - (z_1-z_r) (y_2+\cdots + y_{r-1}) \\
& \geq (r-2) \big((y_r - y_1) z_1 - (z_1 - z_r) y_r\big) \\
& = (r-2) y_1z_r \Big(\frac{y_r}{y_1} - \frac{z_1}{z_r} \Big).
\end{split}
\end{align}
To finish the proof, we need to bound $y_r/y_1$ and $z_1/z_r$. In view of \eqref{eq:intermediate-1}, 
we also have $h_r\geq h_{r-1}\geq\cdots\geq h_2$ and $k_1\leq k_2\leq\cdots\leq k_{r-1}$. Hence,
\begin{equation}\label{eq:intermediate-2}
\frac{y_r}{y_1} = \frac{\lambda (G') y_r}{\lambda (G') y_1} \geq \frac{\lambda (G') y_r}{(n-n_1-(r-1)) h_r}.
\end{equation}
Using the fact that $\lambda (G') (h_r - y_r) = y_1$, we obtain 
\begin{equation}\label{eq:intermediate-3}
\frac{\lambda (G') y_r}{h_r} = \lambda (G') - \frac{y_1}{h_r}.
\end{equation}
It follows from \eqref{eq:intermediate-2} and \eqref{eq:intermediate-3} that 
\begin{equation}\label{eq:intermediate-4}
\frac{y_r}{y_1} \geq \frac{\lambda (G') - \frac{y_1}{h_r}}{n-n_1-(r-1)} > 
\frac{\lambda (G') - 1 - \frac{25r}{\lambda (G')}}{n-n_1-(r-1)},
\end{equation}
where the last inequality follows from Claim \ref{claim:x-comp}. Similarly, 
\begin{align*}
\frac{z_1}{z_r} 
& = \frac{\lambda (H) z_1}{\lambda (H) z_r}
\leq \frac{\lambda (H) z_1}{(n-n_r-(r-1)) k_1}.
\end{align*}
Since $\lambda (H) (k_1 - z_1) = z_r$, we have 
\[
\frac{\lambda (H) z_1}{k_1} = \lambda (H) - \frac{z_r}{k_1}.
\]
Hence, we deduce that
\begin{equation}\label{eq:intermediate-5}
\frac{z_1}{z_r} \leq \frac{\lambda (H) - \frac{z_r}{k_1}}{n-n_r-(r-1)} < 
\frac{\lambda (H) - 1 + \frac{20r}{\lambda (H)}}{n-n_r-(r-1)}.
\end{equation}
Combining \eqref{eq:intermediate-4} and \eqref{eq:intermediate-5} with \eqref{eq:lambda-G-H} gives
\begin{equation}\label{eq:intermediate-6}
(\lambda (G')-\lambda (H)) \bm{y}^{\mathrm{T}}\bm{z}
> (r-2)y_1z_r \bigg(\frac{\lambda (G') - 1 - \frac{25r}{\lambda (G')}}{n - n_1 - (r-1)} - 
\frac{\lambda (H) - 1 + \frac{20r}{\lambda (H)}}{n - n_r - (r-1)}\bigg).
\end{equation}
By Claim \ref{claim:lower-bound-lambda}, $\lambda (G') > \lambda (T_r(n)) - 2r/n$. Since 
$\lambda (H) < \lambda (T_r(n))$, we get
\begin{align*}
\Big(\lambda (G') - 1 - \frac{25r}{\lambda (G')}\Big) 
- \Big(\lambda (H) - 1 + \frac{20r}{\lambda (H)}\Big)
& = \lambda (G') - \lambda (H) - 5r\Big(\frac{5}{\lambda (G')} + \frac{4}{\lambda (H)}\Big) \\
& > -\frac{2r}{n} - \frac{45r}{(r-1)n/r - 1} \\
& > -\frac{100r}{n}.
\end{align*}
Combining this with \eqref{eq:intermediate-6} and noting that $n_1 - n_r = 1$, we find
\begin{align*}
\frac{(\lambda (G')-\lambda (H)) \bm{y}^{\mathrm{T}}\bm{z}}{(r-2)y_1z_r}
& > \frac{\lambda (H) - 1 + \frac{20r}{\lambda (H)} - \frac{100r}{n}}{n - n_r - r} - \frac{\lambda (H) - 1 + \frac{20r}{\lambda (H)}}{n - n_r - (r-1)} \\
& = \frac{\big(\lambda (H) - 1 + \frac{20r}{\lambda (H)}\big) - \frac{100r}{n} (n-n_r-r+1)}{\big(n - n_r - (r-1)\big) (n - n_r - r)} \\
& > 0,
\end{align*}
which yields that $\lambda (G') > \lambda (H)$. This completes the proof of Theorem \ref{thm:co-connected-spectral}.
\hfill $\Box$ \par\vspace{2.5mm}

\section{Disproof of a conjecture of Gregory}

The following conjecture is originally due to Gregory, which was introduced by Cioab\u a \cite{Cioaba2021} 
on the ``Spectral Graph Theory online 2021" organized by Brazilian researchers.
 
\begin{conjecture}[Gregory, c.f.\,~\cite{Cioaba2021}]\label{conj:Gregoy}
Let $G$ be a connected graph with chromatic number $k\geq 2$ and $S$ be an independent set. Then
\[
\sum_{v\in S} x_v^2 \leq \frac{1}{2} - \frac{k-2}{2\sqrt{(k-2)^2 + 4(k-1)(n-k+1)}},
\]
where $x_v$ is the component of the Perron vector of $G$ with respect to the vertex $v$.
\end{conjecture}

In 2010, Cioab\u{a} \cite{Cioaba2010} presented a necessary and sufficient condition for a 
connected graph to be bipartite in terms of its Perron vector $\bm{x}$: Let $S$ be an 
independent set of a connected graph $G$. Then $\sum_{v\in S} x_v^2 \leq 1/2$ with 
equality if and only if $G$ is bipartite having $S$ as one color class. If 
Conjecture \ref{conj:Gregoy} were proven to be true, it would generalize Cioab\u a's 
theorem to a general graph with given chromatic number. This conjecture was also 
listed as Conjecture 23 in \cite{Liu-Ning2023}.

In this section, we disprove Conjecture \ref{conj:Gregoy}. For this purpose, we need 
some additional notations. For a given integer $k$, let $n = (k-1)t + s$, where 
$0\leq s< k-1$. Let $G_{n,k}$ be an $n$-vertex graph with chromatic number $k$, 
whose vertex set can be partitioned into $(t+1)$ sets $V_0, V_1,\ldots, V_t$ such that
$|V_0|=\cdots =|V_{t-1}|=k-1$ and $|V_t|=s$. The edge set of $G_{n,k}$ is 
\[
E(G_{n,k}) = \{uv: u\in V_i, v\in V_{i+1}, i=0,1,\ldots,t-1\} \cup \{uv: u,v\in V_0\}.
\]
That is, $V_0$ forms a clique, while $V_1,V_2,\ldots, V_t$ are independent sets in 
$G_{n,k}$ (see Fig.\,\ref{fig:G-n-k}). 

\begin{figure}[htbp]
\centering
\begin{tikzpicture}[scale=0.9, line width=.6pt]
\coordinate (u1) at (0,0);
\coordinate (u2) at (2,0);
\coordinate (u3) at (3.6,0);
\coordinate (u4) at (6,0);
\coordinate (u5) at (7.6,0);

\coordinate (v2) at (2,0.6);
\coordinate (w2) at (2,-0.6);
\coordinate (v3) at (3.6,0.6);
\coordinate (w3) at (3.6,-0.6);
\coordinate (v4) at (6,0.6);
\coordinate (w4) at (6,-0.6);
\coordinate (v5) at (7.6,0.4);
\coordinate (w5) at (7.6,-0.4);

\filldraw[fill=cyan!25, draw=cyan!25] (u2) ellipse [x radius=0.5,y radius=0.95];
\filldraw[fill=cyan!25, draw=cyan!25] (u3) ellipse [x radius=0.5,y radius=0.95];
\filldraw[fill=cyan!25, draw=cyan!25] (u4) ellipse [x radius=0.5,y radius=0.95];
\filldraw[fill=magenta!25, draw=magenta!25] (u5) ellipse [x radius=0.5,y radius=0.8];

\draw (0.22,-0.9) -- (v2) -- (0,0.9);
\draw (0,-0.9) -- (w2) -- (0.22,0.9);
\draw (v2) -- (v3) -- (w2);
\draw (v2) -- (w3) -- (w2);
\draw (w5) -- (v4) -- (v5);
\draw (w5) -- (w4) -- (v5);

\node[rotate=90] at (u2) {$\cdots$}; 
\node[rotate=90] at (u3) {$\cdots$}; 
\node[rotate=90] at (u4) {$\cdots$}; 
\node[rotate=90] at (u5) {$\cdots$}; 
\node at ($(u3)!.5!(u4)$) {$\cdots$};

\filldraw[fill=blue!20, draw=blue!20] (u1) ellipse [x radius=0.7,y radius=0.95];

\filldraw[fill=blue!20, draw=blue!20] (-0.25,-1.7) rectangle (0.25,-1.2);
\filldraw[fill=cyan!25, draw=cyan!25] (1.75,-1.7) rectangle (2.25,-1.2);
\filldraw[fill=cyan!25, draw=cyan!25] (3.35,-1.7) rectangle (3.85,-1.2);
\filldraw[fill=cyan!25, draw=cyan!25] (5.65,-1.7) rectangle (6.35,-1.2);
\filldraw[fill=magenta!25, draw=magenta!25] (7.35,-1.7) rectangle (7.85,-1.2);

\node at (0,0) {\small $K_{k-1}$};
\node at (0,-1.45) {\small $V_0$};
\node at (2,-1.45) {\small $V_1$};
\node at (3.6,-1.45) {\small $V_2$};
\node at (6,-1.45) {\small $V_{t-1}$};
\node at (7.6,-1.45) {\small $V_t$};

\foreach \i in {v2,w2,v3,w3,v4,w4,v5,w5}
\filldraw (\i) circle (0.07);
\end{tikzpicture}
\caption{The graph $G_{n,k}$}
\label{fig:G-n-k}
\end{figure}
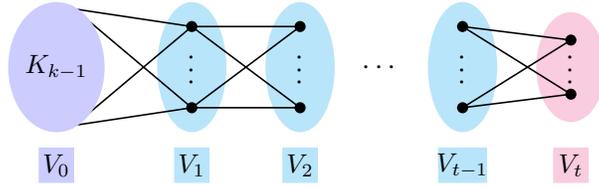

\begin{theorem}\label{thm:counterexample}
Let $n,k$ be two positive integers with $n\geq 2(k-1)$ and $G_{n,k}$ be the graph defined 
as above. There exists an independent set $S$ in $G_{n,k}$ such that 
\[
\sum_{u\in S} x_u^2 = \frac{1}{2} - O\Big(\frac{k^5}{n^3}\Big).
\]
\end{theorem}

\begin{proof}
Let $\bm{x}$ be the Perron vector of $G_{n,k}$, and set $\lambda: = \lambda(G_{n,k})$ 
for brevity. Our proof hinges on the following two claims.

\begin{claim}\label{claim:order-G-n-k}
$2(k-1) - \lambda < \frac{(k-1)^3 \pi^2}{n^2}$.
\end{claim}

\noindent {\it Proof of Claim \ref{claim:order-G-n-k}}.
Define a vector $\bm{z}\in\mathbb{R}^n$ for $G_{n,k}$ such that the components
of $\bm{z}$ within each $V_i$ are equal and denoted by $z_i$, where
\begin{equation}\label{eq:xv}
z_i = \sin\frac{(i+1)\pi}{t+2}, \quad i=0,1,\ldots,t.
\end{equation}
Let $\Delta$ be the maximum degree of $G_{n,k}$. Then $\Delta = 2(k-1)$. 
By simple arguments we find 
\[
(\Delta - \lambda) \cdot \|\bm{z}\|_2^2 
\leq \sum_{v\in V(G_{n,k})} (\Delta - d(v)) z_v^2 + \sum_{uv\in E(G_{n,k})} (z_u - z_v)^2.
\]
Below we shall estimate the terms of the right-hand side of the above inequality.
To find the first term, note that
\begin{align*}
\sum_{v\in V(G_{n,k})} (\Delta - d(v)) z_v^2 
& = (k-1) z_0^2 + s(k-1) z_t^2 \\
& = (k-1)(s+1) z_0^2 \\
& \leq (k-1)^2 z_0^2,
\end{align*}
where the second equality follows from $z_0=z_t$, and the last inequality uses the fact $s\leq k-2$. 
To find the second term, note that
\begin{align*}
\sum_{uv\in E(G_{n,k})} (z_u - z_v)^2
& = s(k-1) (z_{t-1} - z_t)^2 + (k-1)^2 \sum_{j=0}^{t-2} (z_j - z_{j+1})^2 \\
& < (k-1)^2 \sum_{j=0}^{t-1} (z_j - z_{j+1})^2.
\end{align*}
Combining the above three inequalities, gives
\[
(\Delta - \lambda) \cdot \|\bm{z}\|_2^2 
< (k-1)^2 z_0^2 + (k-1)^2 \sum_{j=0}^{t-1} (z_j - z_{j+1})^2.
\]
It follows from $\|\bm{z}\|_2^2 > (k-1)\sum_{j=0}^{t-1} z_j^2$ that 
\begin{equation}\label{eq:Delta-lambda}
\Delta - \lambda 
< \frac{(k-1) z_0^2 + (k-1) \sum_{j=0}^{t-1} (z_j - z_{j+1})^2}{\sum_{j=0}^{t-1} z_j^2}.
\end{equation}

Next, we proceed to evaluate the terms in \eqref{eq:Delta-lambda}. By \eqref{eq:xv},
\begin{align*}
\sum_{j=0}^{t-1} (z_j - z_{j+1})^2 
& = 4\sin^2\frac{\pi}{2(t+2)} \sum_{j=0}^{t-1} \cos^2 \frac{(2j+3)\pi}{2(t+2)} \\
& = \sin^2\frac{\pi}{2(t+2)} \sum_{j=0}^{t-1} \big(\e^{ (2j+3)\pi\imgi/(2t+4) } + \e^{ -(2j+3)\pi\imgi/(2t+4) }\big)^2 \\
& = \sin^2\frac{\pi}{2(t+2)} \sum_{j=0}^{t-1} \big(\e^{ (2j+3)\pi\imgi/(t+2) } + \e^{ -(2j+3)\pi\imgi/(t+2) } + 2\big) \\
& = 2\Big(t - 2\cos\frac{\pi}{t+2}\Big) \sin^2 \frac{\pi}{2(t+2)}.
\end{align*}
This implies that 
\begin{equation}\label{eq:square-yj}
\sum_{j=0}^{t-1} (z_j - z_{j+1})^2 < 2(t-1)\cdot \sin^2 \frac{\pi}{2(t+2)}.
\end{equation}
On the other hand, 
\begin{align*}
\sum_{j=0}^{t-1} z_j^2 
& = \sum_{j=0}^{t-1} \sin^2 \frac{(j+1)\pi}{t+2} \\
& = \frac{t}{2} - \frac{1}{4} \sum_{j=0}^{t-1} \big( \e^{2(j+1)\pi\imgi/(t+2)} + \e^{-2(j+1)\pi\imgi/(t+2)} \big) \\
& = \frac{t+1}{2} + \frac{\cos\frac{2\pi}{t+2}}{2},
\end{align*}
which yields that 
\begin{equation}\label{eq:sum-square-yj}
\sum_{j=0}^{t-1} z_j^2 > \frac{t+1}{2}.
\end{equation}
Substituting \eqref{eq:square-yj} and \eqref{eq:sum-square-yj} into \eqref{eq:Delta-lambda} gives
\[
\Delta - \lambda 
< \frac{(k-1) z_0^2 + 2(t-1)(k-1) \sin^2 \frac{\pi}{2(t+2)}}{(t+1)/2}.
\]
Combining this inequality with $z_0=\sin\frac{\pi}{t+2}$, we obtain
\begin{align*}
\Delta - \lambda 
& < 2(k-1)\cdot \frac{\sin^2\frac{\pi}{t+2} + 2(t-1) \sin^2 \frac{\pi}{2(t+2)}}{t+1} \\
& = 4(k-1) \cdot \frac{\sin^2\frac{\pi}{2(t+2)} \cdot \Big( 2\cos^2 \frac{\pi}{2(t+2)} + (t-1)\Big)}{t+1} \\
& < 4(k-1) \sin^2\frac{\pi}{2(t+2)} \\
& < \frac{(k-1)^3 \pi^2}{n^2}.
\end{align*}
This completes the proof of the claim.
\hfill $\diamondsuit$ \par\vspace{2mm}

By symmetry, the vertices in $V_0$ have the same component of $\bm{x}$, denoted by $x_0$. 

\begin{claim}\label{claim:a}
$x_0 < \frac{(k-1)^2 \pi^2}{n^{3/2}}$.
\end{claim}

\noindent {\it Proof of Claim \ref{claim:a}}.
Since $A(G_{n,k}) \bm{x} = \lambda \bm{x}$, we have 
$\bm{1}^{\mathrm{T}} A(G_{n,k}) \bm{x} = \lambda \bm{1}^{\mathrm{T}} \bm{x}$, 
where $\bm{1}$ is the all-ones vector. Hence,
\[
\sum_{u\in V(G_{n,k})} d(u) x_u = \lambda \sum_{u\in V(G_{n,k})} x_u.
\]
This is equivalent to 
\begin{equation}\label{eq:identity}
\sum_{u\in V(G_{n,k})} \big(2(k-1) - d(u)\big) x_u 
= (2(k-1) - \lambda) \sum_{u\in V(G_{n,k})} x_u.
\end{equation}
Note that the degree of vertex in $V_0$ is $2k-3$. By \eqref{eq:identity} and Cauchy--Schwarz inequality,
\begin{align*}
(k-1) x_0 & < \sum_{u\in V(G_{n,k})} \big(2(k-1) - d(u)\big) x_u \\
& = (2(k-1) - \lambda) \sum_{u\in V(G_{n,k})} x_u \\
& \leq (2(k-1) - \lambda)\sqrt{n}.
\end{align*}
It follows from Claim \ref{claim:order-G-n-k} that
\[
(k-1) x_0 < \frac{(k-1)^3 \pi^2}{n^{3/2}},
\]
completing the proof of Claim \ref{claim:a}.
\hfill $\diamondsuit$ \par\vspace{2mm}

Now, let 
\[ 
S := \bigcup_{i=1}^{\lceil t/2\rceil} V_{2i-1}.
\]
Obviously, $S$ is an independent set in $G_{n,k}$. For any vertex $u\in S$, by eigenvalue\,--\,eigenvector 
equation, we have 
\[
\lambda x_u^2 = \sum_{v\in N(u)} x_ux_v.
\]
Summing this equation for all $u\in S$, we get
\begin{align*}
\lambda \sum_{u\in S} x_u^2
& = \sum_{u\in S} \sum_{v\in N(u)} x_ux_v \\
& = \sum_{ij\in E(G_{n,k})} x_ix_j - \sum_{i,j\in V_0,\, i\neq j} x_ix_j \\
& = \frac{\lambda}{2} - \sum_{i,j\in V_0,\, i\neq j} x_ix_j,
\end{align*}
which implies that 
\begin{equation}\label{eq:sum-x-u-square}
\sum_{u\in S} x_u^2 = \frac{1}{2} - \frac{1}{\lambda} \sum_{i,j\in V_0,\, i\neq j} x_ix_j. 
\end{equation}

Finally, combining \eqref{eq:sum-x-u-square} and Claim \ref{claim:a} implies 
\[
\sum_{u\in S} x_u^2 = \frac{1}{2} - \binom{k-1}{2}\frac{x_0^2}{\lambda} = \frac{1}{2} - O\Big(\frac{k^5}{n^3}\Big),
\]
completing the proof of Theorem \ref{thm:counterexample}.
\end{proof}

\section{Concluding remarks}

In this paper, we present two variations of spectral Tur\'an theorems. The first one strengthens 
a classical result of Nikiforov by offering a more precise  upper bound on the spectral radius 
of graphs. The second variation provides a spectral analogue of Ore's problem, determining the 
unique graph that attains the maximum spectral radius $\spex_{cc}(n, K_{r+1})$ among all 
$n$-vertex co-connected graphs without $K_{r+1}$ as a subgraph. Generally, it would be 
interesting to determine $\spex_{cc}(n, F)$ for other $F$. 

In addition, we also disprove a conjecture of Gregory regarding the Perron vector of connected 
graphs with a given chromatic number. We conclude this paper with a modified version of Gregory's conjecture.

\begin{conjecture}
Let $G$ be a connected graph of order $n$ with chromatic number $k\geq 2$. 
Suppose that $S$ is an independent set, and $\bm{x}$ is the Perron vector of $G$. Then
\[
\sum_{v\in S} x_v^2 < \frac{1}{2} - O\Big(\frac{k^5}{n^3}\Big).
\]
\end{conjecture}

\noindent {\bfseries Note added.} After we finished the first version of this manuscript \cite{Liu-Ning2023-2}, 
Casey Tompkins brought to our attention a related paper \cite{MT23}, in which the authors generalized 
several classical theorems in extremal combinatorics by replacing a global
constraint with a local constraint, and obtained many beautiful theorems.


\begin{thebibliography}{99}
\bibitem{Aouchiche-Hansen2010}
M. Aouchiche, P. Hansen, A survey of automated conjectures in
spectral graph theory, \emph{Linear Algebra Appl.} {\bfseries 432} (2010)
2293--2322.

\bibitem{Bollobas-Nikiforov2007}
B. Bollob\'as, V. Nikiforov, Cliques and the spectral radius, 
\emph{J. Combin. Theory Ser. B} {\bfseries 97}\,(5) (2007) 859--865. 

\bibitem{Bondy-Murty2008}
J.A. Bondy, U.S.R Murty, Graph Theory,
Graduate Texts in Mathematics 244. Springer-Verlag, London, 2008.

\bibitem{BJ08}
N. Bougard, G. Joret, Tur\'an's theorem and $k$-connected graphs, 
\emph{J. Graph Theory} {\bfseries 58}\,(1) (2008) 1--13. 

\bibitem{Brouwer1981} 
A.E. Brouwer, Some lotto numbers from an extension of Tur\'an's theorem, 
Afdeling Zuivere Wiskunde [Department of Pure Mathematics]. vol.\,152, 1981.

\bibitem{Bradac2022} D. Brada\v{c}, A generalization of Tur\'an's theorem,
\href{https://arxiv.org/abs/2205.08923}{arXiv: 2205.08923}, 2022.

\bibitem{Christophe2008} J. Christophe, S. Dewez, J. Doignon, S. Elloumi, 
G. Fasbender, P. Gr\'egoire, D. Huygens, M. Labb\'e, H. M\'elot, H. Yaman, 
Linear inequalities among graph invariants: using GraPHedron to uncover optimal relationships, 
\emph{Networks} {\bfseries 52} (2008) 287--298.

\bibitem{Cioaba2021}
S.M. Cioab\u a, The principal eigenvector of a connected graph, in: Spectral
Graph Theory Online Conference [EB/OL]. 2021. \url{http://spectralgraphtheory.org/
sgt-online-program}.

\bibitem{Cioaba2010}
S.M. Cioab\u{a}, A necessary and sufficient eigenvector condition for a connected graph to be bipartite,
\emph{Electron. J. Linear Algebra} {\bfseries 20} (2010) 351--353.

\bibitem{Cvetkovi1972}
D.M. Cvetkovi\'c, Chromatic number and the spectrum of a graph, 
\emph{Publ. Inst. Math.} (Beograd) {\bfseries 28}\,(14) (1972) 25--38.

\bibitem{Cvetkovi-Doob-Sachs1995}
D.M. Cvetkovi\'c, M. Doob, H. Sachs, Spectra of Graphs\,--\,Theory and Application, 
third ed., Johann Ambrosius Barth Verlag, Heidelberg, Leipzig, 1995.

\bibitem{Edwards-Elphick1983} 
C. Edwards, C. Elphick, Lower bounds for the clique and the chromatic number of a graph, 
\emph{Discrete Appl. Math.} {\bfseries 5} (1983) 51--64.

\bibitem{Elphick-Linze-Wocjan2023}
C. Elphick, W. Linz, P. Wocjan, Two conjectured strengthenings of Turán's theorem, 
\emph{Linear Algebra Appl.} {\bfseries 684} (2023) 23--36.

\bibitem{Gitler-Valencia2013} I. Gitler, C. E. Valencia, On bounds for some graph invariants, 
\href{https://arxiv.org/abs/math/0510387}{arXiv: 0510387}, 2013.

\bibitem{Hoffman1970}
A.J. Hoffman, On eigenvalues and colorings of graphs. Graph Theory and its 
Applications (Proc. Advanced Sem., Math. Research Center, Univ. of Wisconsin, Madison, Wis., 1969), 
pp. 79--91, Academic Press, New York-London, 1970.

\bibitem{Keevash2011} P. Keevash, Hypergraph Tur\'an Problems, 
Surveys in Combinatorics, Cambridge University Press, 2011, pp. 83--140.

\bibitem{Li-Peng2023}
Y. Li, Y. Peng, Refinement on spectral Tur\'an's theorem, 
\emph{SIAM J. Discrete Math.} {\bf 37}\,(4) (2023) 2462--2485.

\bibitem{Li-Feng-Liu2022} 
Y. Li, L. Feng, W. Liu, A survey on spectral conditions for some extremal graph problems, 
\emph{Adv. Math. (China)} {\bfseries 51}\,(2) (2022) 193--258.

\bibitem{Lin-Ning-Wu2021}
H.Q. Lin, B. Ning, B. Wu, Eigenvalues and triangles in graphs, 
\emph{Combin. Probab. Comput.} {\bfseries 30}\,(2) (2021) 258--270.

\bibitem{Liu-Ning2023}
L.L. Liu, B. Ning, Unsolved problems in spectral graph theory, 
\emph{Operations Research Transactions} {\bfseries 27}\,(4)
(2023) 33--60. 

\bibitem{Liu-Ning2023-2}
L.L. Liu, B. Ning, A local version of spectral Tur\'an theorem, 
\href{https://arxiv.org/abs/2312.16138}{arXiv: 2312.16138}, 2023.

\bibitem{Liu-Reiher2021} H. Liu, C. Reiher, M. Sharifzadeh, K. Staden,
Geometric constructions for Ramsey-Tur\'an theory, \href{https://arxiv.org/abs/2103.10423}{arXiv: 2103.10423}, 2021.

\bibitem{MT23}
D. Malec, C. Tompkins, Localized versions of extremal problems, 
\emph{European J. Combin.} {\bfseries 112} (2023), Paper No. 103715, 11 pp. 

\bibitem{Motzkin-Straus1965}
T. Motzkin, E. Straus, Maxima for graphs and a new proof of a theorem of Tur\'an,
\emph{Can. J. Math.} {\bfseries 17} (1965) 533--540.

\bibitem{Mubayi2010}
D. Mubayi, Counting substructures I: Color critical graphs, 
\emph{Adv. Math.} {\bfseries 225} (2010) 2731--2740.

\bibitem{Nikiforov2002}
V. Nikiforov, Some inequalities for the largest eigenvalue of a graph,
\emph{Combin. Probab. Comput.} {\bfseries 11} (2002) 179--189.

\bibitem{Nikiforov2009} 
V. Nikiforov, More spectral bounds on the clique and independence numbers, 
\emph{J. Combin. Theory Ser. B} {\bfseries 99}\,(6) (2009) 819--826.

\bibitem{Nikiforov2020}
V. Nikiforov, Tur\'an's theorem implies Stanley's bound, 
\emph{Discuss. Math. Graph Theory} {\bfseries 40}\,(2) (2020) 601--605.

\bibitem{Nosal1970} E. Nosal, Eigenvalues of Graphs, Master Thesis, University of Calgary, 1970.

\bibitem{Ore1962} O. Ore, Theory of Graphs, American Mathematical Society Colloquium Publications, American Math. Society, Providence, R.\,I., 1962.

\bibitem{Shi-Ma2024}
C. Shi, T.L. Ma,
A note on maximum size of a graph without isolated vertices under the given matching number,
\emph{Appl. Math. Comput.} {\bfseries 460} (2024), Paper No. 128295, 6 pp.

\bibitem{Stanley1987}
R.P. Stanley, A bound on the spectral radius of graphs with $e$ edges, 
\emph{Linear Algebra Appl.} {\bfseries 87} (1987) 267--269. 

\bibitem{Turan1961}
P. Tur\'an, Research problems, 
\emph{Magyar Tud. Akad. Mat. Kutat\'o Int. K\"ozl.} {\bfseries 6} (1961) 417--423.

\bibitem{Wilf1986} 
H. Wilf, Spectral bounds for the clique and independence numbers of graphs, 
\emph{J. Combin. Theory Ser. B} {\bfseries 40} (1986) 113--117.

\bibitem{Y19}
L.T. Yuan, A note on Tur\'an's theorem,
\emph{Appl. Math. Comput.} {\bf 361} (2019), 13--14.

\bibitem{Zhang2023}
S.T. Zhang, On the first two eigenvalues of regular graphs,
\emph{Linear Algebra Appl.} {\bfseries 686} (2024) 102--110.
\end{thebibliography}
\end{document}